\documentclass[a4paper,11pt]{article}
\usepackage{latexsym,epsfig}
\usepackage{amsmath,amssymb,amsthm}
\usepackage{graphicx}
\usepackage{color}
\usepackage{subfig}
\usepackage{listings}
\usepackage{courier}
\usepackage{verbatim}
\usepackage[margin=1.1in]{geometry}
\usepackage{tabularx}
\usepackage{epstopdf}
\usepackage{url}
\usepackage{hyperref}

\newtheorem{thm}{Theorem}[section]
\newtheorem{prop}[thm]{Proposition}
\newtheorem{lem}[thm]{Lemma}
\newtheorem{cor}[thm]{Corollary}

\newtheorem{obs}[thm]{Observation}

\newenvironment{packed_item}{

\begin{itemize}
  \setlength{\itemsep}{1pt}
  \setlength{\parskip}{0pt}
  \setlength{\parsep}{0pt}
}{\end{itemize}}

\begin{document}

\title{
     The maximum size of adjacency-crossing graphs
   }

\author{
	Eyal Ackerman\thanks{Department of Mathematics, Physics and Computer Science,
		University of Haifa at Oranim, 	Tivon 36006, Israel.}\and
	Bal\'azs Keszegh\thanks{Alfréd Rényi Institute of Mathematics and ELTE Eötvös Loránd University, Budapest, Hungary. 
		Research supported by the J\'anos Bolyai Research Scholarship of the Hungarian Academy of Sciences, by the National Research, Development and Innovation Office -- NKFIH under the grant K 132696 and FK 132060, by the \'UNKP-21-5 and \'UNKP-22-5 New National Excellence Program of the Ministry for Innovation and Technology from the source of the National Research, Development and Innovation Fund and by the ERC Advanced Grant ``ERMiD''. This research has been implemented with the support provided by the Ministry of Innovation and Technology of Hungary from the National Research, Development and Innovation Fund, financed under the  ELTE TKP 2021-NKTA-62 funding scheme.}
}

\date{}
\maketitle

\begin{abstract}
An \emph{adjacency-crossing} graph is a graph that can be drawn such that every two edges that cross the same edge share a common endpoint.
We show that the number of edges in an $n$-vertex adjacency-crossing graph is at most $5n-10$.
If we require the edges to be drawn as straight-line segments, then this upper bound becomes $5n-11$.
Both of these bounds are tight.

The former result also follows from a very recent and independent work of Cheong et al.~\cite{cheong2023weakly} who showed that the maximum size of \emph{weakly} and \emph{strongly fan-planar} graphs coincide.
By combining this result with the bound of Kaufmann and Ueckerdt~\cite{KU22} on the size of strongly fan-planar graphs and results of Brandenburg~\cite{Br20} by which the maximum size of adjacency-crossing graphs equals the maximum size of \emph{fan-crossing} graphs which in turn equals the maximum size of weakly fan-planar graphs, one obtains the same bound on the size of adjacency-crossing graphs. However, the proof presented here is different, simpler and direct.
\end{abstract}

\section{Introduction}
\label{sec:intro}

Throughout this paper we consider graphs with no loops or parallel edges, unless stated otherwise.
A \emph{topological graph} is a graph drawn in the plane with its vertices
as points and its edges as Jordan arcs that connect the corresponding points
and do not contain any other vertex as an interior point.
Every pair of edges in a topological graph has a finite number of intersection points,
each of which is either a vertex that is common to both edges,
or a crossing point at which one edge passes from one side of the other edge to its other side.
Note that an edge may not cross itself, since edges are drawn as Jordan arcs.
If every pair of edges intersects in at most one point, then the topological graph is \emph{simple}.
We will be mainly interested in such graphs.
A topological graph is also referred to as a \emph{drawing} of its underlying abstract graph.

\emph{Beyond-planar} graphs have attracted a considerable amount of attention recently.
These are graphs that can be drawn as topological graphs which are not necessarily plane, that is, they may contain crossing edges, however, they avoid or exhibit certain crossing patterns.
For example, \emph{$k$-planar} graphs are graphs that can be drawn as topological graphs in which every edge is crossed at most $k$ times, whereas \emph{$k$-quasiplanar} graphs are graphs that can be drawn without $k$ pairwise crossing edges.
By a slight abuse of terminology, we also refer to the drawings themselves as possessing a given beyond-planarity property, for example, in a $k$-planar topological graph every edge is crossed at most $k$ times.
 
The main property of beyond-planar graphs that has been studied is their \emph{density}, that is, the maximum number of edges as a function of the number of vertices which we denote by $n$.
Other properties that have been investigated include recognizing such graphs (e.g., whether it is possible to decide in polynomial-time if a given graph is $1$-planar) and their hierarchy (e.g., whether every $k$-planar graph is $(k+1)$-quasiplanar).\footnote{The answers to these questions is no~\cite{KM13} and yes~\cite{k-planar-k+1-quasi}).}
See the recent book~\cite{HT20} for a survey on beyond-planar graphs.

\smallskip
A new class of beyond-planar graphs named \emph{fan-planar} graphs was introduced by Kaufmann and Ueckerdt~\cite{KU14} in 2014.
They defined these graphs as graphs that can be drawn as topological graphs which avoid both of the following crossing patterns: (1)~an edge which is crossed by two independent edges, that is, edges not sharing a common endpoint; and (2)~an edge $e$ which is crossed by two edges $e_1$ and $e_2$ that share a common endpoint which appears on different sides of $e$. That is, if $e_1$ and $e_2$ are oriented towards their common endpoint and $e$ is oriented arbitrarily, then one of $e_1$ and $e_2$ crosses $e$ locally from its left side to its right side while the other crosses $e$ from its right side to its left side.
These forbidden crossing patterns are referred to in the literature as \emph{Configuration I} and \emph{Configuration II}, respectively (see Figures~\ref{fig:conf-I} and~\ref{fig:conf-II}).
\begin{figure}[t]
	\centering
	\subfloat[]{\includegraphics[width= 3cm]{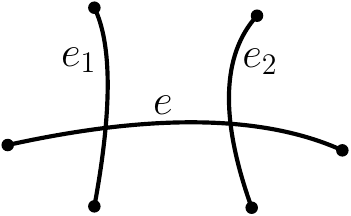}\label{fig:conf-I}}
	\hspace{5mm}
	\subfloat[]{\includegraphics[width= 3cm]{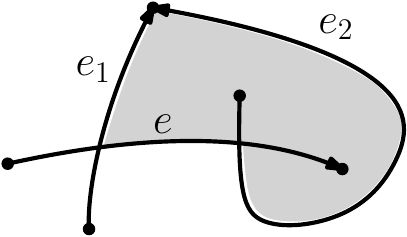}\label{fig:conf-II}}
	\hspace{5mm}
	\subfloat[]{\includegraphics[width= 3.5cm]{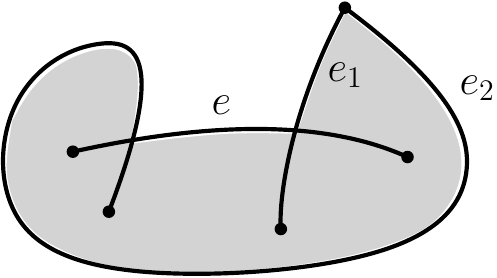}\label{fig:conf-III}}
	\hspace{5mm}
	\subfloat[]{\includegraphics[width= 3cm]{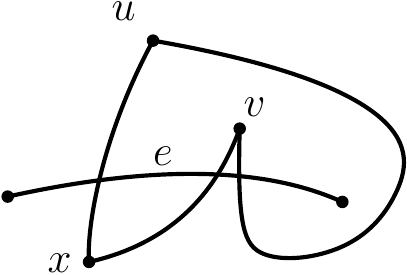}\label{fig:triangle-crossing}}
	\caption{Crossing patterns. (a) Configuration I: $e$ is crossed by $e_1$ and $e_2$ which do not share a common vertex (they may or may not cross). (b) Configuration II: $e$ is crossed by $e_1$ and $e_2$ which share a common vertex. However, if they are oriented towards this vertex and $e$ is oriented arbitrarily, then one of $e_1$ and $e_2$ crosses $e$ from left to right and the other from right to left. Alternatively, one endpoint of $e$ is in the bounded region determined by $e$, $e_1$ and $e_2$. (c) Configuration III: both endpoints of $e$ are in the bounded region determined by $e$, $e_1$ and $e_2$. (d) a triangle-crossing of $e$.} 
	\label{fig:configurations}
\end{figure}

It was claimed in~\cite{KU14} that an $n$-vertex graph that can be drawn as a simple topological graph which avoids Configurations~I and~II has at most $5n-10$ edges (for $n \ge 3$) and that this bound is tight.
Following the introduction of fan-planar graphs in~\cite{KU14}, several other of their combinatorial and computational properties were studied by researchers within the Graph Drawing community (see~\cite{BG20} for a survey).
One such example is a result of Brandenburg~\cite{Br20} by which for every \emph{fan-crossing} topological graph there is a fan-planar topological graph with the same set of vertices and the same number of edges.
A graph is \emph{fan-crossing} if it can be drawn such that for every edge there is a vertex that is incident to all the edges that cross this edge.
A slightly less restrictive crossing pattern requires that every two edges that cross the same edge share a common vertex, that is, only Configuration~I is forbidden.
Graphs admitting such drawings are called \emph{adjacency-crossing} graphs~\cite{Br20}.
Note that an adjacency-crossing topological graph is not necessarily fan-crossing, since it might contain a \emph{triangle-crossing}, that is, an edge $e$ which is crossed by three edges $(u,v)$, $(v,x)$ and $(u,x)$, see Figure~\ref{fig:triangle-crossing} for an example (note that it is not the only possible configuration of a triangle-crossing).
However, it was shown in~\cite{Br20} that every adjacency-crossing graph is fan-crossing, whereas 
there are fan-crossing graphs which are not fan-planar.

It is important to note that all these results refer to simple topological graphs and that when considering drawings using straight-line edges, then all of these classes coincide since Configuration II is impossible.
Klemz et al.~\cite{KKRS23} considered non-simple fan-planar topological graphs and showed that they can be redrawn as simple fan-planar topological graphs without introducing new crossings.
They have also observed a gap in the proof of the density of fan-planar graphs in the preprint~\cite{KU14} that introduced them.
This led to a new definition of fan-planarity in the recent journal version~\cite{KU22} of~\cite{KU14}. Namely, an additional forbidden configuration has been introduced, \emph{Configuration~III}, see Figure~\ref{fig:conf-III}.
Configurations~II and~III can be described as configurations in which an edge $e$ is crossed by two adjacent edges $e_1$ and $e_2$ such that at least one endpoint of $e$ is contained in a bounded region whose boundary consists of segments of $e$, $e_1$ and $e_2$.

In order to distinguish between the two definitions of fan-planar graphs, Cheong et al.~\cite{CPS22} referred to (topological) graphs avoiding Configurations~I, II and~III as \emph{strongly} fan-planar, which implies that graphs that avoid only Configurations~I and~II are \emph{weakly} fan-planar.
Therefore, by~\cite{KU22} the density of strongly fan-planar graphs is $5n-10$, and one can no longer use this bound and the reductions in~\cite{Br20} to get a similar bound on the density of weakly fan-planar graphs, fan-crossing graphs and adjacency-crossing graphs.
Note also that the definition of strongly fan-planar graphs does not generalize to the sphere, since Configurations~I and~III are equivalent on the sphere.
However, drawing weakly fan-planar, fan-crossing and adjacency-crossing graphs on the sphere is equivalent to drawing them in the plane.

\medskip
The main result of this work is a new proof for the claimed density of simple adjacency-crossing topological graphs, that is, graphs avoiding Configuration~I. This immediately implies the same bound for strongly/weakly fan-planar and fan-crossing graphs.
Furthermore, we also verify the conjecture from~\cite{KU22} (see also~\cite{BG20}) that a graph that admits a fan-planar drawing using straight-line edges has at most $5n-11$ edges.
All of these bounds are tight for infinitely many values of $n$ due to matching lower bound constructions~\cite{KU22}.

\begin{thm}\label{thm:no-conf-I}
An $n$-vertex graph that can be drawn as a simple topological graph in which every two edges that cross a third edge share a common vertex has at most $5n-10$ edges, for every $n \ge 3$.
If the graph can be drawn in such a way using straight-line edges, then it has at most $5n-11$ edges.
\end{thm}

\begin{cor}\label{cor:all}
If $G$ is an $n$-vertex graph that can be drawn as a simple topological graph which is fan-crossing, weakly fan-planar or strongly fan-planar, then $G$ has at most $5n-10$ edges, for every $n \ge 3$.
If $G$ can be drawn in such a way using straight-line edges, then it has at most $5n-11$ edges.
\end{cor}

While compiling the latest version of this paper, we learned that Cheong et al.~\cite{cheong2023weakly} have recently showed that although there are weakly fan-planar graphs which are not strongly fan-planar, the densities of these graph classes are equal.
Therefore, their work combined with the results in~\cite{KU22} and~\cite{Br20} also implies the first parts of Theorem~\ref{thm:no-conf-I} and Corollary~\ref{cor:all}.
However, the proof presented here is still of interest since it uses different methods, is simpler and self-contained.

\smallskip
For convenience, the definitions of the classes of the topological graphs mentioned above are summarized in the following table:

\smallskip
\begin{tabular}{|l|l|}
	\hline
	Graph class &  Forbidden configuration(s) \\
	\hline
	strongly fan-planar &  Configurations I, II and III\\
	\hline
	weakly fan-planar &  Configurations I and II\\
	\hline
	fan-crossing &  Configuration I and triangle-crossings \\ 
	\hline
	adjacency-crossing &  Configuration I \\
	\hline 
\end{tabular}

\smallskip
It follows from these definitions and the results in~\cite{Br20} and~\cite{cheong2023weakly} that the hierarchy of these classes is: $\textrm{strongly fan-planar} \subsetneq \textrm{weakly fan-planar} \subsetneq \textrm{fan-crossing} = \textrm{adjacency-crossing}$.

\section{Proof of Theorem~\ref{thm:no-conf-I}}
\label{sec:proof}

For a topological graph $G$ we denote by $M(G)$ the plane map induced by $G$.
That is, the vertices of $M(G)$ are the vertices and crossing points in $G$
and the edges of $M(G)$ are the crossing-free segments of the edges of $G$,
where each such edge-segment connects two vertices of $M(G)$.
We will use capital letters to denote the vertices of $G$, and small letters to denote
crossing points in $G$ (that are vertices in $M(G)$).
An edge of $M(G)$ will usually be denoted by its endpoints, e.g., $xy$,
whereas for an edge of $G$ we will use the standard notation, e.g., $(A,B)$.

Let $G$ be a simple topological graph with $n\geq 3$ vertices such that any two edges that cross the same edge share a common endpoint.
We prove that such graphs have at most $5n-10$ edges by induction on $n$.
For $3 \le n \leq 8$ we have $5n-10 > {{n}\choose{2}}$ and thus the claim trivially holds.
Therefore, we may assume that $n \geq 9$.
Furthermore, we may assume that $G$ is connected and that the degree of every vertex in $G$ is at least $6$,
for otherwise the claim easily follows by induction.
The claim is also easy to prove if the vertex-connectivity of $M(G)$ is less than two.

\begin{prop}\label{prop:2-connected}
If $M(G)$ is not $2$-connected, then $G$ has at most $5n-10$ edges.
\end{prop}

\begin{proof}
Assume that $M(G)$ has a vertex $x$ such that $M(G) \setminus \{x\}$ is not connected.
The vertex $x$ is either a vertex of $G$ or a crossing point of two of its edges.
Suppose that $x$ is a vertex of $G$. Then, $G \setminus \{x\}$ is also not connected.
Let $G_1,\ldots,G_k$ be the connected components of $G \setminus \{x\}$,
let $G'$ be the topological graph induced by $V(G_1) \cup \{x\}$
and let $G''$ be the topological graph induced by $V(G_2) \cup \ldots \cup V(G_k) \cup \{x\}$.
Note that $7 \leq |V(G')|,|V(G'')| < n$, since $\delta(G) \geq 6$.
Therefore, it follows from the induction hypothesis that $|E(G)| \leq 5|V(G')|-10 + 5|V(G'')|-10 = 5(n+1)-20 < 5n-10$.

Suppose now that $x$ is a crossing point of two edges $e_1$ and $e_2$.
Let $G_1,\ldots,G_k$ be the connected components of $G \setminus \{e_1,e_2\}$,
let $G'$ be the topological graph induced by $V(G_1)$
and let $G''$ be the topological graph induced by $V(G_2) \cup \ldots \cup V(G_k)$.
Note that $6 \leq |V(G')|,|V(G'')| < n$, since $\delta(G) \geq 6$ ($e_1$ and $e_2$ may not share a vertex since they cross).
Therefore, it follows from the induction hypothesis that $|E(G)| \leq 5|V(G')|-10 + 5|V(G'')|-10 + 2 = 5n-18 < 5n-10$.
\end{proof}

In light of Proposition~\ref{prop:2-connected}, we may assume henceforth that $M(G)$ is $2$-connected.
Therefore the boundary of every face $f$ is a closed Jordan curve which corresponds to a simple cycle in $M(G)$.
We refer to the vertices and the edges of $M(G)$ on this cycle as being \emph{incident} to $f$ or simply being the vertices and edges of $f$.
The \emph{size} of a face $f$, $|f|$, is the number of edges of $M(G)$ on its boundary.

\medskip
We will use the \emph{Discharging Method} to prove Theorem~\ref{thm:no-conf-I}.
This technique, that was introduced and used successfully for proving structural properties of planar graphs
(most notably, in the proof of the Four Color Theorem~\cite{AH77}),
has proven to be a useful tool also for showing upper bounds on the size of some beyond-planar graphs~\cite{Ac09,Ac14,Ac19,AT07,AFKMT12,min-k-planar,KP11}.
In our case, we begin by assigning a \emph{charge} to every face of the planar map $M(G)$ such that the total charge is $4n-8$.
Then, we redistribute the charge in several steps such that eventually the charge of every face is non-negative
and the charge of every vertex $A \in V(G)$ is $0.4\deg(A)$.
Hence, $0.8|E(G)| = \sum_{A \in V(G)} 0.4\deg(A) \leq 4n-8$ and we get the desired bound on $|E(G)|$.
We now describe the proof in details. 

\paragraph{Charging.}
Let $V'$, $E'$, and $F'$ denote the vertex, edge, and face sets of $M(G)$, respectively.
For a face $f \in F'$ we denote by $V(f)$ the set of vertices of $G$ that are incident to $f$.
It is easy to see that $\sum_{f \in F'} |V(f)| = \sum_{A \in V(G)} \deg(A)$ and that
$\sum_{f \in F'} |f| = 2|E'| = \sum_{u \in V'} \deg(u)$.
Note also that every vertex in $V' \setminus V(G)$ is a crossing point in $G$
and therefore its degree in $M(G)$ is four. Hence,
$$
    \sum_{f \in F'} |V(f)| = \sum_{A \in V(G)} \deg(A) =
        \sum_{u \in V'} \deg(u) - \sum_{u \in V' \setminus V(G)}\deg(u) =
    2|E'| - 4\left(|V'|-n\right).
$$
Assigning every face $f \in F'$ a charge of $|f|+|V(f)|-4$,
we get that the total charge over all faces is
$$
    \sum_{f \in F'}\left(|f|+|V(f)|-4\right) = 2|E'|+ 2|E'| - 4\left(|V'|-n\right) - 4|F'| = 4n-8,
$$
where the last equality follows from Euler's Polyhedral Formula by which $|V'|+|F'|-|E'| = 2$
(recall that $G$ is connected and therefore so is $M(G)$).

\paragraph{Discharging.}
We will redistribute the charges in several steps.
We denote by $ch_i(x)$ the charge of an element $x$ (either a face in $F'$ or a vertex in $V(G)$) after the $i$th step,
where $ch_0(\cdot)$ represents the initial charge function.
We will use the terms \emph{triangles}, \emph{quadrilaterals}, \emph{pentagons} and \emph{hexagon}
to refer to faces of size $3$, $4$, $5$ and $6$, respectively.
An integer before the name of a face denotes the number of original vertices on its boundary.
For example, a $2$-triangle is a face of size $3$ with exactly two original vertices on its boundary.
Since $G$ is a simple topological graph, there are no faces of size $2$ in $F'$.
Because $G$ does not contain three pairwise crossing edges there are no $0$-triangles and therefore the initial charge of every face is non-negative.

\medskip\noindent\textbf{Step 1: Charging the vertices of $G$.}
In the first discharging step, every face sends $0.4$ units of charge to every original vertex of $G$ on its boundary, see Figure~\ref{fig:step1}. $\hfill\leftrightsquigarrow$\medskip

\begin{obs}
For every vertex $A \in V(G)$ we have $ch_1(A) = 0.4\deg(A)$.
For every face $f \in F'$ we have $ch_1(f) \ge 0$ unless $f$ is a $1$-triangle in which case $ch_1(f) = -0.4$.
\end{obs}

In order to describe the way the charge of $1$-triangles becomes non-negative,
we will need the following definitions.
Let $f$ be a face, let $e$ be one of its edges, and let $f'$ be the other face that shares $e$ with $f$.
We say that $f'$ is the \emph{edge-neighbor} of $f$ at $e$.
If the boundaries of two faces intersect at a vertex $v \in V' \setminus V(G)$ and they are not edge-neighbors at any of the edges that are incident to $v$, then we say that these two faces are \emph{vertex-neighbors} at $v$.
Note that a face cannot be a vertex- or an edge-neighbor of itself since $M(G)$ is $2$-connected.

\smallskip\noindent\textbf{Wedge-neighbors.}
Let $f_0$ be a $1$-triangle in $M(G)$ and let $A \in V(G)$ and $x_1,y_1 \in V' \setminus V(G)$ be its vertices.
Denote by $e_x$ (resp., $e_y$) the edge of $G$ that contains $x_1$ (resp., $y_1$) and $A$.
Let $f_1$ be the edge-neighbor of $f_0$ at $x_1y_1$.
For $i \geq 1$, if $f_{i}$ is a $0$-quadrilateral,
then denote by $x_{i+1}y_{i+1}$ the edge of $M(G)$ opposite to $x_iy_i$ in $f_{i}$,
such that $e_x$ contains $x_{i+1}$ and $e_y$ contains $y_{i+1}$,
and let $f_{i+1}$ be the edge-neighbor of $f_i$ at $x_{i+1}y_{i+1}$.
Observe that $f_i \neq f_j$ for $i < j$, for otherwise $x_j$ coincides with one of $x_i$ and $x_{i+1}$
(which implies that $e_x$ crosses itself) or with one of $y_i$ and $y_{i+1}$ (which implies that $e_x$ and $e_y$ intersect more than once).
Furthermore, there is no edge of $G$ that contains two different edges $x_iy_i$ and $x_jy_j$, for such an edge would cross each of $e_x$ and $e_y$ more than once.
Let $j$ be the maximum index for which $f_j$ is defined. 
Since there are at most $|E(G)|-2$ edges that cross $e_x$ and $e_y$ such an index $j$ must exist.
We say that $f_j$ is the \emph{wedge-neighbor} of $f_0$ at $x_1y_1$ and that  $f_0$ is the wedge-neighbor of $f_j$ at $x_jy_j$ (see Figure~\ref{fig:step2}).
The \emph{wedge} of $f_0$ consists of the faces $f_0,\ldots,f_j$ and its \emph{apex} is $A$.
\begin{figure}[t]
	\centering
	\subfloat[Step 1: every face sends $0.4$ units of charge to each vertex of $G$ on its boundary]{\includegraphics[width= 4cm]{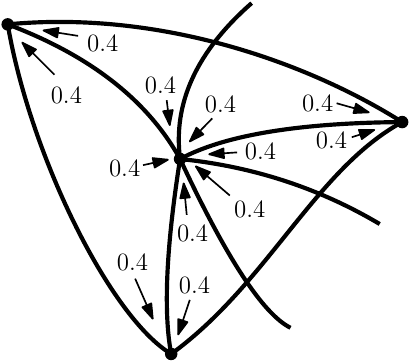}\label{fig:step1}}
	\hspace{5mm}
	\subfloat[Step 2: the $1$-triangle $f_0$ receives $0.4$ units of charge from its wedge-neighbor $f_j$.]{\includegraphics[width= 8cm]{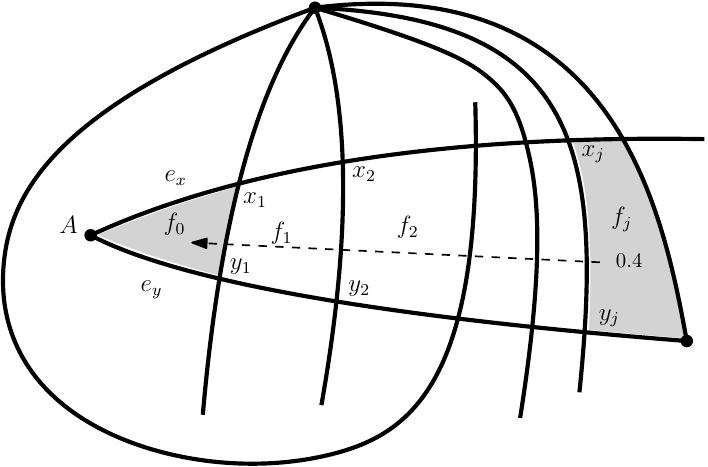}\label{fig:step2}}
	\caption{The first two discharging steps.} 
	\label{fig:step1-2}
\end{figure}

Observe that since the relations being an edge-neighbor at a certain edge of $M(G)$
and being an opposite edge in a $0$-quadrilateral are both one-to-one,
it follows that indeed there cannot be another triangle besides $f_0$ that is a wedge-neighbor of $f_j$ at $x_jy_j$.
Note also that since $e_x$ and $e_y$ already intersect at $A$ (and because by definition $f_j$ cannot be a $0$-quadrangle),
either $|f_j| \geq 5$ or $|f_j|=4$ and $|V(f_j)| \geq 1$.

\begin{obs}\label{obs:one-wedge-neighbor}
Let $f\in F'$ be a face and let $e$ be one of its edges.
Then there is at most one triangle $t$ such that $t$ is a wedge-neighbor of $f$ at $e$.
If such a triangle exists, then either $|f| \geq 5$ or $|f|=4$ and $|V(f)| \geq 1$.
\end{obs}

Recall that the only faces with a negative charge are the $1$-triangles.
In order to charge them, we discharge every wedge-neighbor of a $1$-triangle.

\medskip\noindent\textbf{Step 2: Charging $1$-triangles.}
Let $f_0$ be a $1$-triangle, let $x_1,y_1$ be its vertices which are crossing points, and let $f_j$ be
its wedge-neighbor at $x_1y_1$ as defined above.
Then $f_j$ sends $0.4$ units of charge to $f_0$ \emph{through} $x_jy_j$ (see Figure~\ref{fig:step2}). $\hfill\leftrightsquigarrow$\medskip

Note that if a face $f$ contributes charge through one of its edges $xy$ in Step~2, then $x,y \in V' \setminus V(G)$. Therefore, the number of such edges on the boundary of $f$ is at most $|f|-|V(f)|-1$ if $0 < |V(f)| < |f|$.
Considering this fact, the discharging steps and an easy case analysis we get:

\begin{obs}\label{obs:ch_2}
	Let $f \in F'$ be a face. If $|V(f)|=|f|$, then $ch_2=1.6|f|-4$. If $|V(f)|=0$, then $ch_2(f) \ge 0.6|f|-4$ .Otherwise, $ch_2(f) \ge 0.6|f|+|V(f)|-3.6$. Specifically,
	\begin{packed_item}
		\vspace{-3mm}
		\item if $f$ is a $1$-triangle, then $ch_2(f)=0$;
		\item if $f$ is a $2$-triangle, then $ch_2(f)=0.2$; 
		\item if $f$ is a $0$-quadrilateral, then  $ch_2(f)=0$;
		\item if $f$ is a $1$-quadrilateral, then $ch_2(f) \geq -0.2$;
		\item if $f$ is a $0$-pentagon, then $ch_2(f) \geq -1$;
		\item if $f$ is a $0$-hexagon, then $ch_2(f) \geq -0.4$; and
		\item if $f$ is none of the above, then $ch_2(f) \ge 0$.
	\end{packed_item}
\end{obs}

The remaining discharging steps will thus handle $1$-quadrilaterals, $0$-pentagons and $0$-hexagons with a negative charge.
In the third discharging step some of these $1$-quadrilaterals are charged.

\medskip
\noindent\textbf{Step 3: discharging through common edges.} 
Let $f \in F'$ be a face such that $ch_2(f) > 0$ and let $Q_f$ be the (multi)set of $1$-quadrilaterals that are edge-neighbors of $f$ at one of its edges which is incident to a vertex in $V(f)$ and whose charge after Step~2 is negative
(if a $1$-quadrilateral in $Q_f$ has more than one such common edge with $f$, then it appears with multiplicity in $Q_f$). 
Then $f$ sends $\min\{0.1,ch_2(f)/|Q_f|\}$ units of charge to each face in $Q_f$ through the edge that they share (see Figure~\ref{fig:step3}).$\hfill\leftrightsquigarrow$
\medskip

\begin{figure}[t]
	\centering
	\subfloat[Step 3: $f$ sends $0.1$ units of charge to the $1$-quadrilateral which is its edge-neighbor at $Ax$.]{\includegraphics[width= 5.5cm]{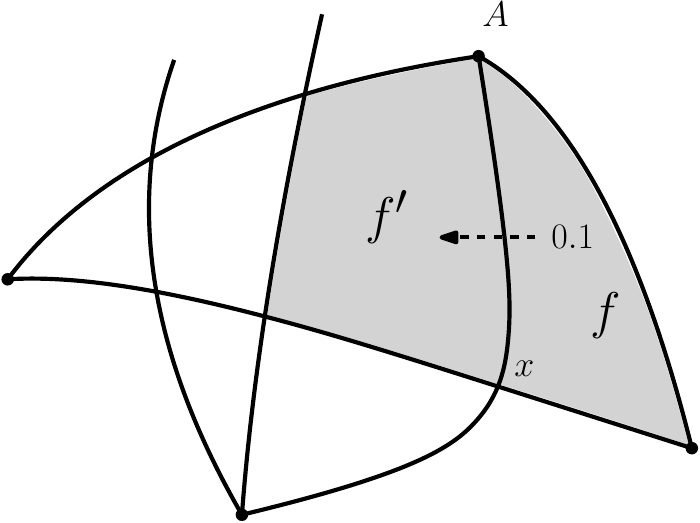}\label{fig:step3}}
	\hspace{5mm}
	\subfloat[Step 4: $f$ sends charge to the $0$-pentagon which is its vertex-neighbor at $x$.]{\includegraphics[width= 6cm]{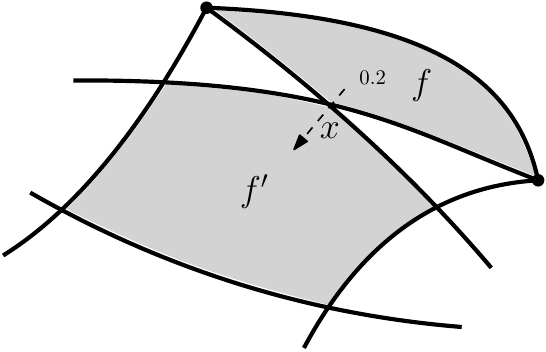}\label{fig:step4}}
	\caption{Discharging Steps 3 and 4.} 
	\label{fig:step3-4}
\end{figure}

This discharging rule was phrased such that it is clear that a face cannot contribute more than its excess charge. However, we claim that $f$ contributes exactly $0.1$ units of charge through each of the above edges.

\begin{prop}\label{prop:at-least-0.1}
	Let $f$ be a $1$-quadrilateral such that $ch_2(f)<0$, let $Ax$ be an edge of $f$ such that $V(f)=\{A\}$ and let $f'$ be the edge-neighbor of $f$ at $Ax$. 
	Then $f'$ contributes $0.1$ units of charge to $f$ in Step~3, unless $f'$ is a $1$-quadrilateral such that $ch_2(f')<0$ or a $1$-triangle.
\end{prop}

\begin{proof}
	By Observation~\ref{obs:ch_2} we have $ch_2(f') \ge 0.6|f'|+|V(f')|-3.6$ (note that $f'$ has at least one vertex in $V(G)$ and at least one vertex in $V' \setminus V(G)$).
	Clearly, $|Q_{f'}| \le |f'|$.
	Therefore, if $(0.6|f'|+|V(f')|-3.6)/|f'| \ge 0.1$, then $f$ receives $0.1$ units of charge from $f'$. This inequality holds when $|f'| \ge 6$ (recall that $|V(f')| \ge 1$), or if $f'$ is a $k$-pentagon for $k \ge 2$ or a $3$-quadrilateral (recall that $f'$ has at least one vertex which is a crossing point, so it cannot be a $4$-quadrilateral, a $3$-triangle, etc.).
	If $f'$ is a $1$-pentagon, a $1$-quadrilateral with a non-negative charge after Step~2 or a $2$-triangle, then $ch_2(f') \ge 0.2$ and $f'$ has at most two edge-neighbors that are $1$-quadrilaterals. 
	Therefore $f$ receives from $f'$ $0.1$ units of charge in Step~3 in each of these cases.
	Finally, if $f'$ is a $2$-quadrilateral, then $ch_2(f') \ge 0.8$ and $|Q_{f'}|\le 4$, thus $f'$ contributes to $f$ $0.1$ units of charge.
\end{proof}

The fourth discharging step handles $0$-pentagons and $0$-hexagons with a negative charge and possibly some of the $1$-quadrilaterals with a negative charge.

\medskip\noindent\textbf{Step 4: discharging through common vertices.} 
Let $f \in F'$ be a face such that $ch_3(f) > 0$ and let $S_f$ be the (multi)set of faces that are vertex-neighbors of $f$ and whose charge after Step~3 is negative (if a face is a vertex-neighbor of $f$ at more than one vertex, then it appears with multiplicity in $S_f$).
Then $f$ sends $\min\{0.2,ch_3(f)/|S_f|\}$ units of charge to each face in $S_f$ through the vertex that they share (see Figure~\ref{fig:step4}).$\hfill\leftrightsquigarrow$
\medskip

We will show that after the fourth discharging step the charge of every $0$-pentagon and $0$-hexagon is non-negative.
Before doing so, we first introduce some useful notations and propositions. 
When considering a $0$-$k$-gon $f$ (where $k=5$ or $6$),
we denote the vertices of $f$ by $v_0,\ldots,v_{k-1}$ listed in a clockwise cyclic order.
The edges of $f$ are $e_i=v_{i-1}v_i$, for $i=0,\ldots,k-1$ (addition and subtraction are modulo $k$).
For every edge $e_i=v_{i-1}v_i$ we denote by $(A_i,B_i)$ the edge of $G$ that contains $e_i$,
such that $v_{i-1}$ is between $A_i$ and $v_{i}$ on $(A_i,B_i)$.
Denote by $t_i$ the $1$-triangle which is a wedge-neighbor of $f$ at $e_i$, if such a $1$-triangle exists.
Note that if $t_i$ exists, then one of its vertices and the apex of the corresponding wedge is $A_{i+1}=B_{i-1}$ (that is, they represent the same point).
Finally, we denote by $f_i$ the vertex-neighbor of $f$ at $v_i$.
See Figure~\ref{fig:0-pentagon-A4-B2} for an example of these notations. 
Note that different names may refer to the same point, however, since $G$ is a simple topological graph we have:

\begin{obs}\label{obs:A_i,B_i}
Let $f$ be a $0$-$k$-gon for $k=5,6$.
For every $i=0,\ldots,k-1$ $v_i$ is a crossing point of $(A_i,B_i)$ and $(A_{i+1},B_{i+1})$, therefore $\{A_{i},B_{i}\} \cap \{A_{i+1},B_{i+1}\} = \emptyset$.
\end{obs}

The following will also be useful later.

\begin{obs}\label{obs:common-vertex}
Let $f$ be a $0$-$k$-gon for $k=5,6$.
Then for every $i=0,\ldots,k-1$ the edges $(A_i,B_i)$ and $(A_{i+2},B_{i+2})$ both cross $(A_{i+1},B_{i+1})$ and therefore $\{A_i,B_i\} \cap \{A_{i+2},B_{i+2}\} \ne \emptyset$.
\end{obs}

\begin{prop}\label{prop:A4=B2}
Let $f$ be a $0$-pentagon and suppose that $t_i$ and $t_{i+1}$ exist, for some $0 \le i \le 4$.
Then $A_{i-1} = B_{i+2}$.
\end{prop}

\begin{proof}
Assume without loss of generality that $i=0$.
The existence of the $1$-triangles $t_0$ and $t_1$ as wedge-neighbors of $f$ at $e_0$ and $e_1$, respectively, implies that $B_4=A_1$ is the apex of the wedge of $t_0$ and that $B_0=A_2$ is the apex of $t_1$, see Figure~\ref{fig:0-pentagon-A4-B2}.
\begin{figure}[t]
	\centering
	\includegraphics[width= 7cm]{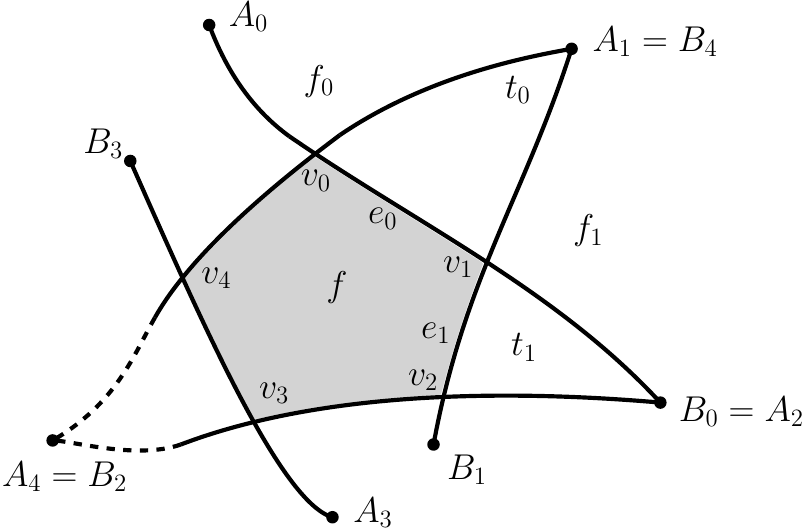}
	\caption{A $0$-pentagon $f$ and a corresponding notation of its vertices, vertex-neighbors, edges and the edges of $G$ that contain its edges. If $t_0$ and $t_1$ exist, then $A_4=B_2$.}
	\label{fig:0-pentagon-A4-B2}		
\end{figure}
It follows from Observation~\ref{obs:common-vertex} that $\{A_2,B_2\} \cap \{A_4,B_4\} \ne \emptyset$.
By Observation~\ref{obs:A_i,B_i} we have $B_4 = A_1 \ne A_2,B_2$ and $A_2 = B_0 \ne A_4, B_4$.
Therefore $A_4 = B_2$. 
\end{proof}

\begin{prop}\label{prop:0-pentagon-wedge-neighbors}
Let $f$ be a $0$-pentagon and suppose that $t_i$ and $t_{i+1}$ exist, for some $0 \le i \le 4$.
Then $t_i$ and $t_{i+1}$ are the edge-neighbors of $f$ at $e_i$ and $e_{i+1}$, respectively.
\end{prop}

\begin{proof}
Assume without loss of generality that $i=0$.
The existence of the $1$-triangles $t_0$ and $t_1$ implies that $B_4=A_1$ is the apex of the wedge of $t_0$ and that $B_0=A_2$ is the apex of $t_1$.
By Proposition~\ref{prop:A4=B2} we also have $A_4=B_2$.
By symmetry it is enough to show that $t_0$ is the edge-neighbor of $f$ at $e_0$.

\smallskip
Suppose for contradiction it is not, that is, the wedge of $t_0$ contains at least one $0$-quadrilateral.
Let $(A,B)$ be the edge of $G$ that contains the edge of $t_0$ which is not incident to $A_1$. 
Note that $(A,B)$ crosses $(A_1,B_1)$ and $(A_4,B_4)$ at vertices of $t_0$.
Since $(A,B)$ and $(A_0,B_0)$ both cross $(A_1,B_1)$, they must share a vertex.
Suppose first that this common vertex is $A_0$ and assume without loss of generality that $A=A_0$.
Since $(A_2,B_2)$ also crosses $(A_1,B_1)$ we have $\{A_0,B\} \cap \{A_2,B_2\} \ne \emptyset$.
Clearly $A_0 \ne A_2 = B_0$ since $G$ has no loops and $A_0 \ne B_2=A_4$ by Observation~\ref{obs:A_i,B_i}.
If $B=A_2=B_0$, then $(A,B)$ and $(A_0,B_0)$ are parallel edges in $G$.
Therefore $B=B_2$. However, this implies that $(A,B)$ intersects $(A_4,B_4)$ twice: once at a vertex of $t_0$ and once at $B=B_2=A_4$, see Figure~\ref{fig:0-pentagon1} for an illustration.
\begin{figure}[t]
	\centering
	\subfloat[If $A=A_0$ is the common vertex of $(A,B)$ and $(A_0,A_2)$, then $B=B_2$ and $(A,B)$ and $(A_4,B_4)$ intersect twice.]{\includegraphics[width= 7cm]{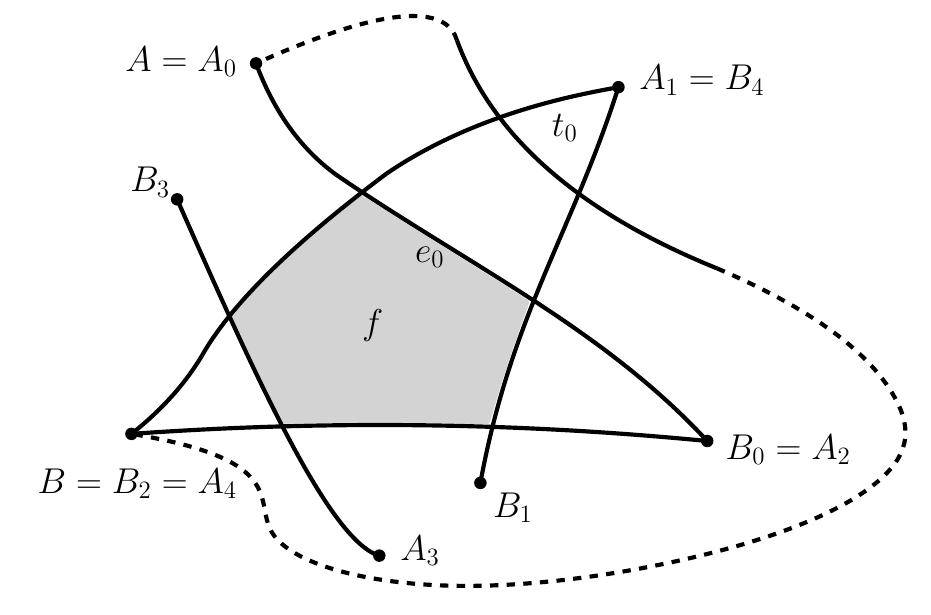}\label{fig:0-pentagon1}}
	\hspace{5mm}
	\subfloat[If $B=A_2$ is the common vertex of $(A,B)$ and $(A_0,A_2)$ and $A=A_3$ then $(A,B)$ and $(A_1,B_1)$ intersect twice.]{\includegraphics[width= 7cm]{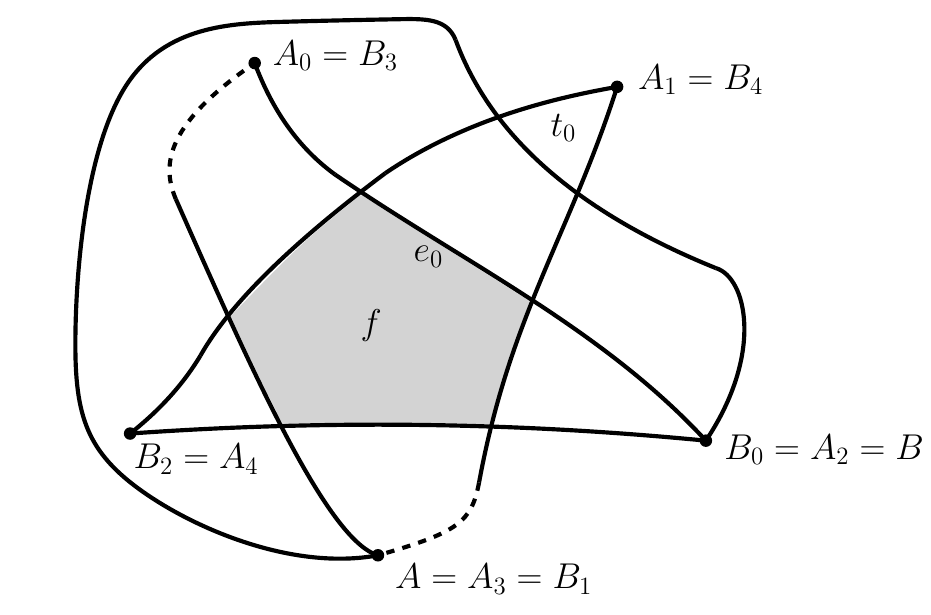}\label{fig:0-pentagon2}}
	\hspace{5mm}
	\subfloat[If $B=A_2$ is the common vertex of $(A,B)$ and $(A_0,A_2)$ and $A=B_3$ then $(A,B)$ and $(A_1,B_1)$ intersect twice.]{\includegraphics[width= 7cm]{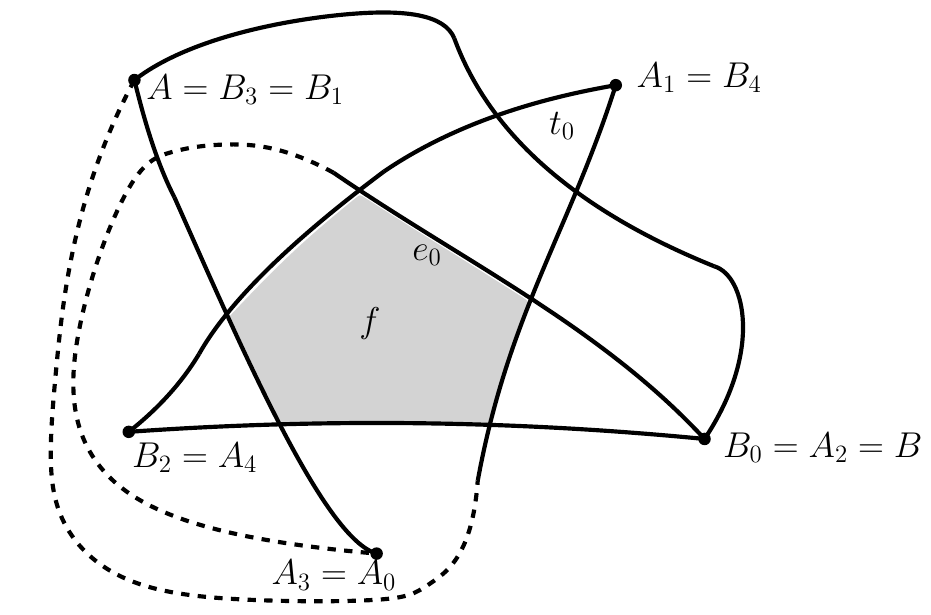}\label{fig:0-pentagon3}}
	\caption{Illustrations for the proof of Proposition~\ref{prop:0-pentagon-wedge-neighbors}: $f$ is a $0$-pentagon whose wedge-neighbors at $e_0$ and $e_1$ are $1$-triangles such that $t_0$ is not its edge-neighbor at $e_0$.} 
	\label{fig:0-pentagon}
\end{figure}

\smallskip
Suppose now that the common vertex of $(A,B)$ and $(A_0,A_2)$ is $A_2$.
Assume without loss of generality that $B=A_2=B_0$.
Since both $(A_3,B_3)$ and $(A,B)$ cross $(A_4,B_4)$ we have $\{A_3,B_3\} \cap \{A,B\} \ne \emptyset$.
By Observation~\ref{obs:A_i,B_i} $B=A_2 \ne A_3,B_3$, thus $A=A_3$ or $A=B_3$.

Suppose that $A=A_3$. 
By Observation~\ref{obs:common-vertex} $\{A_0,B_0\} \cap \{A_3,B_3\} \ne \emptyset$. 
It follows from Observation~\ref{obs:A_i,B_i} that $B_0=A_2 \ne A_3,B_3$, thus $A_0=A_3$ or $A_0=B_3$.
The first case $A_0=A_3=A$ was already considered, therefore we may assume that $A_0=B_3$.
By Observation~\ref{obs:common-vertex} $\{A_1,B_1\} \cap \{A_3,B_3\} \ne \emptyset$. 
It follows from Observation~\ref{obs:A_i,B_i} that $A_1 = B_4 \ne A_3,B_3$ and $B_1 \ne A_0=B_3$, thus $B_1=A_3$.
However, then $(A,B)$ intersects $(A_1,B_1)$ twice: at a vertex of $t_0$ and at $A=A_3=B_1$, see Figure~\ref{fig:0-pentagon2} for an illustration.

It remains to consider the case $A=B_3$.
By Observation~\ref{obs:common-vertex} $\{A_0,B_0\} \cap \{A_3,B_3\} \ne \emptyset$.
It follows from Observation~\ref{obs:A_i,B_i} that $B_0=A_2\ne A_3,B_3$, thus $A_0=A_3$ or $A_0=B_3$.
If $A_0=B_3=A$, then $(A,B)$ and $(A_0,B_0)$ are parallel edges which is impossible.
Therefore, $A_0=A_3$.
It follows from Observation~\ref{obs:common-vertex} that $\{A_1,B_1\} \cap \{A_3,B_3\} \ne \emptyset$.
By Observation~\ref{obs:A_i,B_i} $A_3=A_0 \ne A_1,B_1$ and $B_3 \ne A_1=B_4$, thus, $B_3=B_1$.
However, then $(A_1,B_1)$ and $(A,B)$ intersect twice: at a vertex of $t_0$ and at $A=B_3=B_1$,
see Figure~\ref{fig:0-pentagon3} for an illustration.
\end{proof}

\begin{prop}\label{prop:0-hexagon-wedge-neighbors}
Let $f$ be a $0$-hexagon which is a wedge-neighbor of a $1$-triangle at each of its edges.
Then $f$ is the edge-neighbor of $t_i$ at $e_i$ for $i=0,\ldots,5$.
\end{prop}

\begin{proof}
Since $f$ is a wedge-neighbor of a $1$-triangle at each of its edges, it follows that $A_i=B_{i+4}$ for every $i=0,\ldots,5$.
Suppose by contradiction and without loss of generality that $t_0$ is not an edge-neighbor of $f$ at $e_0$ and let $(A,B)$ be the edge of $G$ that contains the edge of $t_0$ which is not incident to $A_1$, see Figure~\ref{fig:0-hexagon}.
\begin{figure}[t]
	\centering
	\includegraphics[width= 7cm]{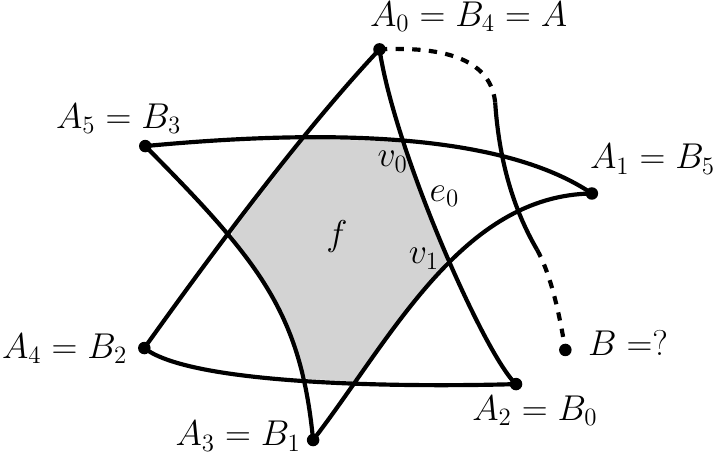}
	\caption{If $f$ is a $0$-hexagon whose wedge-neighbors are all $1$-triangles, then they are also its edge-neighbors.}
	\label{fig:0-hexagon}		
\end{figure}
Since $(A,B)$ and $(A_0,A_2)$ both cross $(A_1,A_3)$ we have $\{A,B\} \cap \{A_0,A_2\} \ne \emptyset$.
Assume without loss of generality that $A=A_0$.
Then $B \ne A_2, A_4$ since $G$ contains no parallel edges.
However, $\{A_0,B\} \cap \{A_2,A_4\} \ne \emptyset$ since the corresponding edges both cross $(A_1,A_3)$.
Thus $A_0=A_2$ or $A_0=A_4$ which implies that $G$ contains a loop which is impossible.
\end{proof}

We can now show that the charge of every $0$-pentagon and $0$-hexagon is non-negative after Step~4.

\begin{prop}\label{prop:0.2-in-Step-4}
	Let $f \in F'$ be a face such that $ch_3(f) < 0$. 
	If $f$ contributes charge in Step~2 through two consecutive edges $e_{i-1}=v_{i-1}v_{i}$ and $e_{i}=v_{i}v_{i+1}$ such that its edge-neighbors at these edges are $1$-triangles, then $f$ receives $0.2$ units of charge from its vertex-neighbor at $v_{i}$ in Step~4.
\end{prop}

\begin{proof}
	Let $t_j$ be the edge-neighbor of $f$ at $e_j$, for $j=i-1,i$,
	and let $f_i$ be the vertex-neighbor of $f$ at $v_{i}$.
	The $1$-triangles $t_{i-1}$ and $t_i$ are also edge-neighbors of $f_{i}$.
	It follows that  $f_{i}$ is incident to at least two original vertices.
	If $f_i$ is a $2$-triangle, then $ch_3(f_i) = 0.2$ and $f$ receives this excess charge in Step~4.
	Otherwise, note that $ch_3(f_i) \ge |f_i|+|V(f_i)|-4-0.4|V(f_i)|-0.4(|f_i|-2)=0.6|f_i|+0.6|V(f_i)|-3.2$, since $f_i$ does not contribute charge through the edges it shares with $t_{i-1}$ and $t_i$.
	This excess charge is contributed to at most $|f_i|-|V(f_i)|$ vertex-neighbors of $f_i$ in Step~4, including $f$.
	Since $\frac{0.6|f_i|+0.6|V(f_i)|-3.2}{|f_i|-|V(f_i)|} \ge 0.2$ for $|f_i| \ge 4$ and $|V(f_i)| \ge 2$ the claim holds.
\end{proof}

\begin{cor}\label{cor:Step4}
	For every face $f \in F'$ we have $ch_4(f) \ge 0$, unless $f$ is a $1$-quadrilateral such that $ch_2(f) < 0$.
\end{cor}

\begin{proof}
	By Observation~\ref{obs:ch_2} and the fact that every face may contribute only its excess charge in Steps~3 and~4 it is enough to consider $0$-pentagons and $0$-hexagons whose charge is negative after Step~3.
	
	Let $f$ be a $0$-pentagon such that $ch_3(f)<0$.
	Then $f$ contributed charge to exactly three, four or five $1$-triangles in Step~2 and ends up with a charge of $-0.2$, $-0.6$ or $-1$, respectively.
	In the first case, there are necessarily two consecutive edges on the boundary of $f$ through which $f$ contributes charge in Step~2.
	It follows from Propositions~\ref{prop:0-pentagon-wedge-neighbors} and~\ref{prop:0.2-in-Step-4} that there is a vertex-neighbor of $f$ that sends $0.2$ units of charge to $f$ in Step~4 and thus $ch_4(f)=0$.
	By similar arguments $f$ receives enough charge to compensate for its negative charge in the other two cases as well.
	If $f$ is a $0$-hexagon such that $ch_2(f) <0$, then it follows from Propositions~\ref{prop:0-hexagon-wedge-neighbors} and~\ref{prop:0.2-in-Step-4} that $f$ receives $0.2$ units of charge from six vertex-neighbors and thus $ch_4(f) \ge 0$.
\end{proof}

It remains to take care of $1$-quadrilaterals whose charge after Step~4 is negative.
To this end we introduce the following definition.
Let $f_0\in F'$ be a face and let $x_0$, $x_1$ and $A$ be three consecutive vertices on the boundary of $f_0$ such that $x_1 \in V' \setminus V(G)$ and $A \in V(G)$ ($x_0$ may be either a crossing point or an original vertex of $G$).
Let $e$ be the edge of $G$ that contains $x_0x_1$.
Suppose that there are faces $f_1,f_2,\ldots,f_k, f'_1,f'_2,\ldots,f'_k$ such that (1)~$f_i$ is a $1$-triangle whose vertices are $x_{i},x_{i+1},A$ for $i=1,2,\ldots,k$; (2)~$f'_i$ is a $0$-quadrilateral which is the edge-neighbor of $f_i$ at $x_ix_{i+1}$ for $i=1,2,\ldots,k-1$; and (3)~$f'_k$ is not a $0$-quadrilateral.
Then $f'_k$ is the \emph{distant-neighbor} of $f_0$ at $x_1$, see Figure~\ref{fig:distant-neighbor} for an example.
\begin{figure}
	\centering
	\includegraphics[width= 5cm]{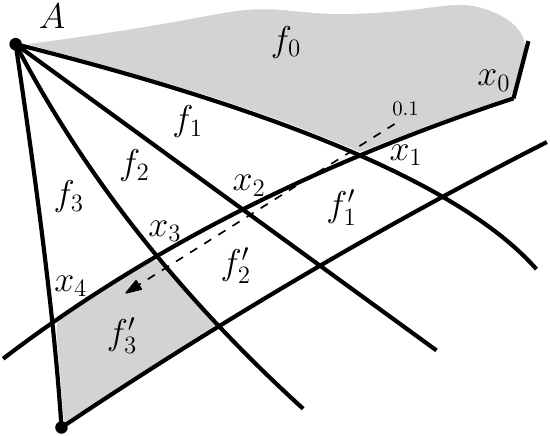}
	\caption{Step 5: $f_0$ sends charge to $f'_3$ -- its distant-neighbor at $x_1$.}
	\label{fig:distant-neighbor}
\end{figure}
Note that $f_0$ may have at most two distant-neighbors at each of its vertices (if $x_0 \in V(G)$ then it is possible that $f$ has another distant-neighbor at $x_1$, one which is incident to the edge of $G$ that contains $Ax_1$).

\medskip\noindent\textbf{Step 5: charging $1$-quadrilateral with a negative charge.} 
Let $f \in F'$ be a face such that $ch_4(f) > 0$ and let $S'_f$ be the (multi)set of faces that are distant-neighbors of $f$ and whose charge after Step~4 is negative.
Then $f$ sends $ch_4(f)/|S'_f|$ units of charge to each face in $S'_f$
(see Figure~\ref{fig:distant-neighbor} for an example).$\hfill\leftrightsquigarrow$
\medskip

We will need the following in order to prove that $1$-quadrilaterals end up with a non-negative charge after Step~5.

\begin{prop}\label{prop:neighboring-1-quadrilaterals}
Let $f_1$ and $f_2$ be two $1$-quadrilaterals that are edge-neighbors at $Ax$ (where $A \in V(G)$)
and contribute charge in Step~2 through edges $xy_1$ and $xy_2$ to $1$-triangles $t_1$ and $t_2$, respectively.
Let $(A,B)$ be the edge of $G$ that contains $Ax$ and let $(A,C_i)$ be the edge of $G$ that contains the other edge of $f_i$ which is incident to $A$, for $i=1,2$.
Then $(C_1,C_2)$ is the edge of $G$ that contains $xy_1$ and $xy_2$ (refer to Figure~\ref{fig:1-quad-1-quad}).
Furthermore, $f_i$ and $t_i$ are edge-neighbors at $xy_i$, for $i=1,2$.
\begin{figure}[t]
	\centering
	\includegraphics[width= 6cm]{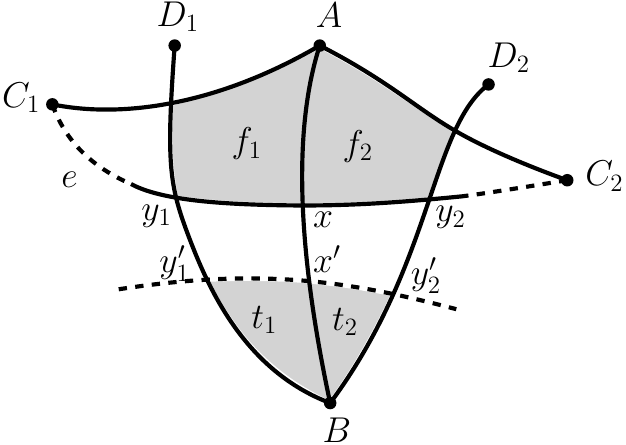}
	\caption{If $f_1$ and $f_2$ contribute charge in Step~2 through $xy_1$ and $xy_2$ to $t_1$ and $t_2$, then $(C_1,C_2)$ contains $xy_1$ and $xy_2$ and $t_1$ and $t_2$ are the immediate neighbors of $f_1$ and $f_2$ at $xy_1$ and $xy_2$, respectively.}
	\label{fig:1-quad-1-quad}		
\end{figure}
\end{prop}

\begin{proof}
Note that $C_1 \ne C_2$ since there are no parallel edges in $G$ and $\deg(A) > 2$.
Since $f_i$ contributes charge through $xy_i$ in Step~2, it is part of a wedge whose apex is $B$.
Denote by $(B,D_i)$ the other edge of $G$ (besides $(A,B)$) that bounds this wedge.
Let $e$ be the edge of $G$ that contains $xy_1$ and $xy_2$. 
Then we wish to show that $e=(C_1,C_2)$.
Since both $e$ and $(A,C_i)$ cross $(B,D_i)$, they must share an endpoint.
It cannot be $A$ since $e$ also crosses $(A,B)$.
Therefore $C_i$ is an endpoint of $e$, for $i=1,2$, that is, $e=(C_1,C_2)$.

As for the second part of the claim, suppose without loss of generality that the edge-neighbor of $f_1$ at $xy_1$ is not $t_1$ but a $0$-quadrilateral and let $x'y'_1$ be its edge which is opposite to $xy_1$ such that $x'$ lies on $(A,B)$.
Then the edge-neighbor of $f_2$ at $xy_2$ must also be a $0$-quadrilateral.
Let $x'y'_2$ be its edge which is opposite to $xy_2$, see Figure~\ref{fig:1-quad-1-quad}.
If we remove the edge $(C_1,C_2)$ which contains $xy_1$ and $xy_2$,
then we get two neighboring $1$-quadrilaterals that would contribute charge in Step~2 to $t_1$ and $t_2$ through $x'y'_1$ and $x'y'_2$, respectively.
As before, this implies that the endpoints of the edge of $G$ that contains $x'y'_1$ and $x'y'_2$ are $C_1$ and $C_2$, however, $G$ does not contain parallel edges.
Therefore $f_i$ is an edge-neighbor of $t_i$ at $xy_i$, for $i=1,2$.
\end{proof}

\begin{lem}\label{lem:after-step-5}
After Step~5 the charge of every face is non-negative.
\end{lem}

\begin{proof}
By Observation~\ref{obs:ch_2}, Corollary~\ref{cor:Step4} and the fact that each face contributes only its excess charge in Steps~3, 4 and 5, it is enough to consider $1$-quadrilaterals that contribute charge to two $1$-triangles in Step~2.
Let $f$ be such a $1$-quadrilateral and let $A,x,y_1,z$ be the vertices incident to $f$ listed in a cyclic order such that $A \in V(G)$.
Then $f$ contributes charge in Step~2 to two $1$-triangles, $t_1$ and $t_2$, through $xy_1$ and $zy_1$, respectively.
Denote by $B$ (resp., $C$) the apex of wedge of $t_1$ (resp., $t_2$) and by $(C,D)$ (resp., $(B,H)$) the edge of $G$ that contains $xy_1$ (resp., $zy_1$).
Let $f'$ and $f''$ be the edge-neighbors of $f$ at $Ax$ and $Az$, respectively.
See Figure~\ref{fig:1-triangle-1-quad} for an illustration and note that $f' \ne f''$ since $\deg(A) > 3$ and $G$ is $2$-connected.

It follows from Proposition~\ref{prop:at-least-0.1} that it is enough to consider the case that at least one of $f'$ and  $f''$ is a $1$-triangle or a $1$-quadrilateral whose charge is negative after Step~2.
We consider the possible cases in the following propositions.

\begin{prop}\label{prop:two-1-triangles}
If $f'$ and $f''$ are $1$-triangles, then $ch_4(f) \ge 0$.
\end{prop}

\begin{proof}
Let $(A,I)$ be the edge of $G$ that contains the other edge of $f'$ besides $Ax$ which is incident to $A$,
see Figure~\ref{fig:two-1-triangles}.
\begin{figure}
	\centering
	\subfloat[If $f'$ is a $1$-triangle, then there exists a triangle-crossing of $(C,D)$]{\includegraphics[width= 6cm]{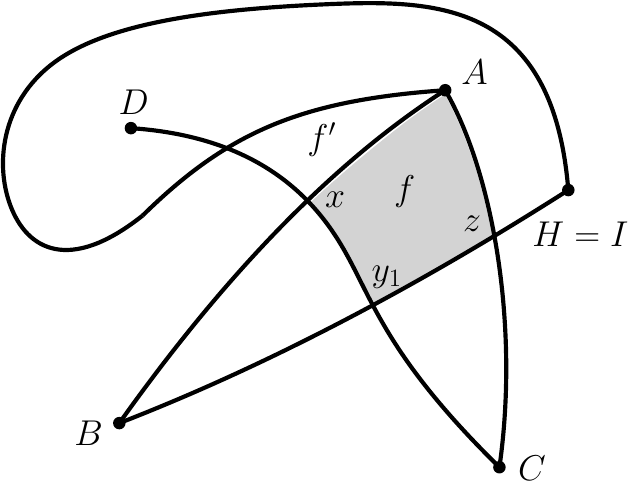}\label{fig:two-1-triangles}}
	\hspace{5mm}
	\subfloat[If $f'$ is a $1$-triangle and $f''$ is a $1$-quadrilateral such that $ch_2(f'')<0$, then $(A,I)$ and $(A,J)$ are parallel edges.]{\includegraphics[width= 7cm]{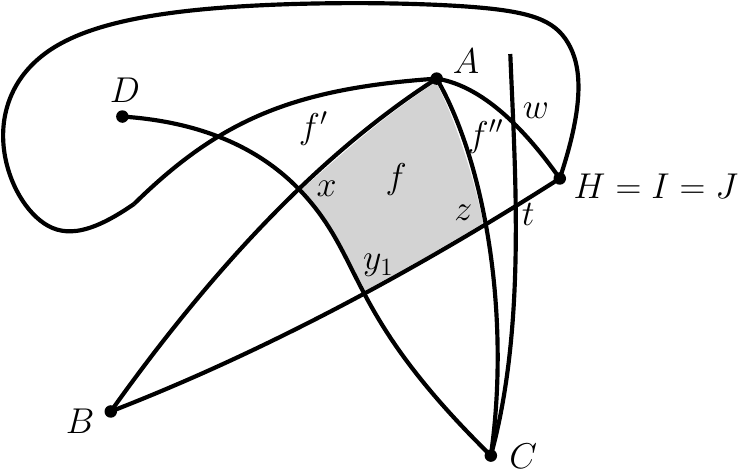}\label{fig:1-triangle-1-quad}}
	\caption{Illustrations for Propositions~\ref{prop:two-1-triangles} and~\ref{prop:1-triangle-1-quadrilateral}. $f$ is a $1$-quadrilateral such that $ch_3(f)<0$.}
	\label{fig:1-triangle-1-triangle/quad}
\end{figure}
Note that $(A,I)$, $(A,B)$ and $(B,H)$ all cross $(C,D)$, therefore $H=I$.
Since there is a triangle-crossing of $(C,D)$, it is not crossed by any other edge.
It follows that $z,y_1,C$ are the vertices of $t_2$.
Similarly, since $f''$ is a $1$-triangle it follows that $x,y_1,B$ are the vertices $t_1$.
Therefore, by Proposition~\ref{prop:0.2-in-Step-4} $f$ receives $0.2$ units of charge from its edge-neighbor at $y_1$ in Step~4 and thus $ch_4(f) \ge 0$.
\end{proof}

\begin{prop}\label{prop:1-triangle-1-quadrilateral}
	It is impossible that one of $f'$ and $f''$ is a $1$-triangle and the other is a $1$-quadrilateral whose charge after Step~2 is negative.
\end{prop}

\begin{proof}
Suppose by contradiction and without loss of generality that $f'$ is a $1$-triangle and $f''$ is a $1$-quadrilateral such that $ch_2(f'') < 0$.
Let $(A,I)$ be the edge of $G$ that contains the other edge of $f'$ besides $Ax$ which is incident to $A$.
As in the proof of Proposition~\ref{prop:two-1-triangles} we conclude that $I=H$.
Let $A,w,t,z$ be the vertices of $f''$ such that $Aw$ is an edge of $f''$, and let $(A,J)$ be the edge of $G$ that contains $Aw$, see Figure~\ref{fig:1-triangle-1-quad}.
Note that $(A,I)$ and $(A,J)$ are different edges, since $\deg(A) > 3$.
However, since $f''$ contributes charge through $wt$ in Step~2, it follows that the apex of the corresponding wedge is $J=H=I$. 
Hence $(A,J)$ and $(A,I)$ are parallel edges which is impossible.
\end{proof}

\begin{prop}\label{prop:two-1-quads}
If $f'$ and $f''$ are $1$-quadrilaterals whose charge is negative after Step~2, then $ch_4(f) \ge 0$.
\end{prop}

\begin{proof}
It follows from Proposition~\ref{prop:neighboring-1-quadrilaterals} (applied for $f$ and $f'$ and for $f$ and $f''$) that the edge-neighbors of $f$ at $xy_1$ and $zy_1$ are $t_1$ and $t_2$, respectively.
Therefore, by Proposition~\ref{prop:0.2-in-Step-4} $f$ receives $0.2$ units of charge from its vertex-neighbor at $y_1$ in Step~4 and ends up with a non-negative charge.
\end{proof}

It remains to consider the case that exactly one of $f'$ and $f''$ is either a $1$-triangle or a $1$-quadrilateral whose charge after Step~2 is negative.
Suppose without loss of generality that $f'$ is such a face and $f''$ is not.
Then by Proposition~\ref{prop:at-least-0.1} $f$ receives $0.1$ units of charge from $f''$ in Step~3, therefore, it is enough to show that $f$ receives $0.1$ units of charge in Step~4 or Step~5.

\begin{prop}\label{prop:f1-or-distant}
	If $ch_3(f)=-0.1$, the edge-neighbor of $f$ at $xy_1$ is $t_1$ and every edge which crosses $(C,D)$ is incident to $B$, then $ch_5(f) \ge 0$.
	Likewise, if $ch_3(f)=-0.1$, the edge-neighbor of $f$ at $zy_1$ is $t_2$ and every edge which crosses $(B,H)$ is incident to $C$, then $ch_5(f) \ge 0$.
\end{prop}

\begin{proof}
	By symmetry it enough to consider the case that $t_1$ is the edge-neighbor of $f$ at $xy_1$ and every edge which crosses $(C,D)$ is incident to $B$.
		
	Consider the vertex-neighbor of $f$ at $y_1$, denoted by $f_1$, and observe that $By_1$ is one of its edges since $t_1$ is the edge-neighbor of $f$ at $xy_1$.
	We claim that unless $f_1$ is a $1$-triangle, it contributes at least $0.1$ units of charge to $f$ in Step~4.
	
	Observe first that $f_1$ cannot be a $1$-quadrilateral.
	Indeed, if $f_1$ is a $1$-quadrilateral, then the edge of $G$ that contains the edge of $f_1$ which is opposite to $By_1$ crosses $(C,D)$ and therefore must be adjacent to $B$.
	Let $(B,P)$ denote this edge and let $(B,Q)$ be the edge of $G$ that contains the other edge of $f_1$ which is adjacent to $B$.
	But then $(B,P)$ and $(B,Q)$ intersect more than once since they cross at the opposite vertex to $y_1$ in $f_1$, see Figure~\ref{fig:after-step5f} for an illustration.
\begin{figure}
	\centering
	\subfloat[If $f_1$ is a $1$-quadrilateral, then $(B,P)$ and $(B,Q)$ intersect more than once.]{\includegraphics[width= 5cm]{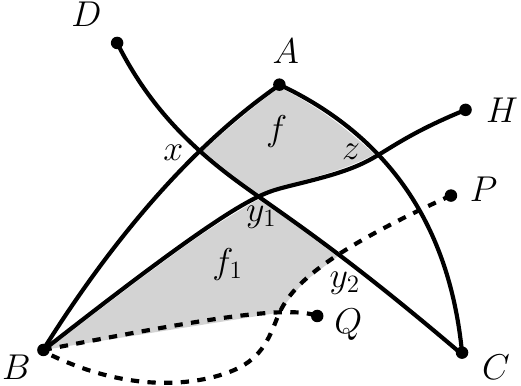}\label{fig:after-step5f}}
	\hspace{5mm}
	\subfloat[If $f_1$ is a $1$-triangle, then $f$ is a distant-neighbor of $f_k$ and receives from it charge in Step~5.]{\includegraphics[width= 7cm]{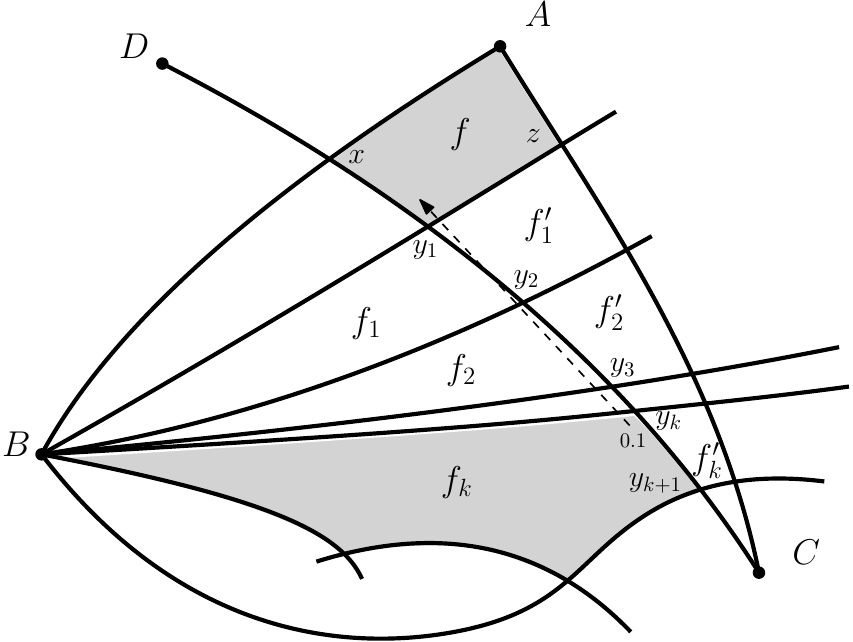}\label{fig:after-step5g}}
	\caption{Illustrations for the proof of Proposition~\ref{prop:f1-or-distant}.}
	\label{fig:after-step5}		
\end{figure}
	
	Note that $f_1$ does not contribute charge through its edges that are adjacent to $y_1$. 
	Indeed, its edge-neighbor at $By_1$ is the $1$-triangle $t_1$.
	Let $y_1y_2$ be its other edge which is adjacent to $y_1$ and note that this edge is a segment of $(C,D)$.
	If $y_2=C$, then the edge-neighbor of $f_1$ at $y_1y_2$ is also a $1$-triangle ($t_2$).
	Otherwise, let $(B,P)$ denote the edge that crosses $(C,D)$ at $y_2$ (recall that one of its endpoint must be $B$), see Figure~\ref{fig:after-step5f}.
	If $f_1$ contributes charges through $y_1y_2$, then $P=H$ which implies that $(B,H)$ and $(B,P)$ are parallel edges.
	
	\smallskip
	Recall that we wish to show that $f_1$ contributes charge to $f$ in Step~4, unless $f$ is a $1$-triangle.
	If $f_1$ is a $2$-triangle, then $ch_3(f_1) = 0.2$ and $f$ receives this excess charge in Step~4.
	Otherwise, if $f_1$ is not a triangle, then note that $ch_3(f_1) \ge |f_1|+|V(f_1)|-4-0.4|V(f_1)|-0.4(|f_1|-2)=0.6|f_1|+0.6|V(f_1)|-3.2$, since $f_1$ does not contribute charge through $By_1$ and $y_1y_2$.
	This excess charge is contributed to at most $|f_1|-|V(f_1)|-1$ vertex-neighbors of $f_1$ in Step~4, including $f$ (note that the vertex-neighbor of $f_1$ at $y_2$ is part of the wedge of $t_2$ and is therefore a $0$-quadrilateral or a $1$-triangle and thus does not get any charge from $f_1$).
	Since $\frac{0.6|f_1|+0.6|V(f_1)|-3.2}{|f_1|-|V(f_1)|-1} \ge 0.1$ in case $|f_1| \ge 5$ and $|V(f_1)| \ge 1$ and also in case $|f_1|=4$ and $|V(f_1)| \ge 2$, it follows that $f$ indeed receives $0.1$ units of charge from $f_1$ in Step~4.
	
	\smallskip
	Finally, suppose that $f_1$ is a $1$-triangle, let $B,y_1,y_2$ be its vertices, let $f'_1$ and $f_2$ be its edge-neighbors at $y_1y_2$ and $By_2$, respectively, and let $f'_2$ be the vertex-neighbor of $f_1$ at $y_2$, see Figure~\ref{fig:after-step5g}.
	Note that $f'_1$ must be a $0$-quadrilateral since it lies in the wedge of $t_2$.
	If $f_2$ is also a $1$-triangle whose vertices are $B,y_2,y_3$, then let $f_3$ be its edge-neighbor at $y_2y_3$ and let $f'_3$ be the vertex-neighbor of $f_2$ at $y_3$.
	Similarly, as long as $f_i$ is a $1$-triangle we define $y_{i+1}$, $f_{i+1}$ and $f'_{i+1}$ in this manner until eventually $f_k$ is not a $1$-triangle for some $k$.
	Note that $f$ is a distant-neighbor of $f_k$ at $y_k$, see Figure~\ref{fig:after-step5g}.
	
	\smallskip
	We claim that $f_k$ contributes at least $0.1$ units of charge to $f$ in Step~5.
	Indeed, one can conclude that $f_k$ cannot be a $1$-quadrilateral using the same arguments that showed that $f_1$ cannot be a $1$-quadrilateral (by removing the edges containing $By_2,\ldots,By_{k-1}$ the face $f_k$ `becomes' $f_1$).
	If $f_k$ is a $2$-triangle, then it does not contribute charge through $By_k$ and $Cy_k$, since its edge-neighbors at these edges are $1$-triangles.
	Furthermore, its vertex-neighbor at $y_k$ is a $0$-quadrilateral so it does not contribute any charge at Step~4 and thus $ch_4(f_k)=0.2$.
	Since $f_k$ has at most one other distant-neighbor besides $f$ it follows that $f_k$ send at least $0.1$ units of charge to $f$ in Step~5.
	
	Otherwise, if $f_k$ is not a $2$-triangle then let $y_ky_{k+1}$ be the other edge on the boundary of $f_k$ besides $By_k$ which is incident to $y_k$ and let $t$ be the number of vertices on the boundary of $f_k$ through which it contributed charge in Step~4.
	Clearly, $f_k$ does not contribute charge through $By_k$ and contributes at most $0.1$ units of charge through its other edge which is incident to $B$.
	It does not contribute charge through $y_ky_{k+1}$ neither, for otherwise the edges of $G$ that contain $By_k$ and $By_{k+1}$ would be parallel.
	Therefore, $ch_4(f_k) \ge |f_k|+|V(f_k)|-4-0.4|V(f_k)|-0.4(|f_k|-3)-0.1-0.2t=0.6|f_k|+0.6|V(f_k)|-0.2t-2.9$.
	Recall that $f_k$ may contribute charge to at most two distant-neighbors through each incident vertex $v$, such that $v \in V' \setminus V(G)$ and $v$ is not one of the $t$ vertices through which $f$ contributes charge in Step~4.
	Furthermore, if $f_k$ contributes charge to two distant-neighbors through $v$, then the two vertices on the boundary of $f_k$ that are adjacent to $v$ must be vertices of $G$.
	Therefore, if $y_k \in V(G)$, that is $y_k=C$, then the extra charge at $f_k$ is contributed to at most $2(|f_k|-|V(f_k)|-t)$ distant-neighbors including $f$.
	Since $\frac{0.6|f_k|+0.6|V(f_k)|-0.2t-2.9}{2(|f_k|-|V(f_k)|-t)} \ge 0.1$ for $|f_k| \ge 4$ and $|V(f_k)| \ge 2$, we have $ch_5(f) \ge 0$.
	Otherwise, if $y_k \in V' \setminus V(G)$ then $f_k$ has at most one distant-neighbor at each of $y_k$ and $y_{k+1}$.
	Therefore, its extra charge is contributed to at most $2(|f_k|-|V(f_k)|-t-2)+2$ distant-neighbors including $f$ in Step~5.
	Since $\frac{0.6|f_k|+0.6|V(f_k)|-0.2t-2.9}{2(|f_k|-|V(f_k)|-t-1)} \ge 0.1$ for $|f_k| \ge 4$ and $|V(f_k)|\ge 2$ and also for $|f_k| \ge 5$ and $|V(f_k)| \ge 1$, we have $ch_5(f) \ge 0$.
\end{proof}

Using this proposition we now consider the remaining cases where $f'$ is either a $1$-triangle or a $1$-quadrilateral with a negative charge after Step~2 and $f''$ is neither of those.

\begin{prop}\label{prop:f'-1quad}
If $f'$ is a $1$-quadrilateral such that $ch_2(f) < 0$ and $f''$ is neither such a $1$-quadrilateral nor a $1$-triangle, then $ch_5(f) \ge 0$.
\end{prop}

\begin{proof}
It follows from Proposition~\ref{prop:neighboring-1-quadrilaterals} that $f'$ has an edge, denote it by $xw$, which is contained in $(C,D)$ and that $\{B,x,y_1\}$ and $\{B,x,w\}$ are the vertex sets of the $1$-triangles to which $f$ and $f'$ contribute charge through $xy_1$ and $xw$, respectively. 
Let $(B,J)$ be the edge of $G$ that contains $Bw$, see Figure~\ref{fig:f''-quad}.
\begin{figure}
	\centering
	\subfloat[$f'$ is a $1$-quadrilateral such that $ch_2(f')<0$.]{\includegraphics[width= 5cm]{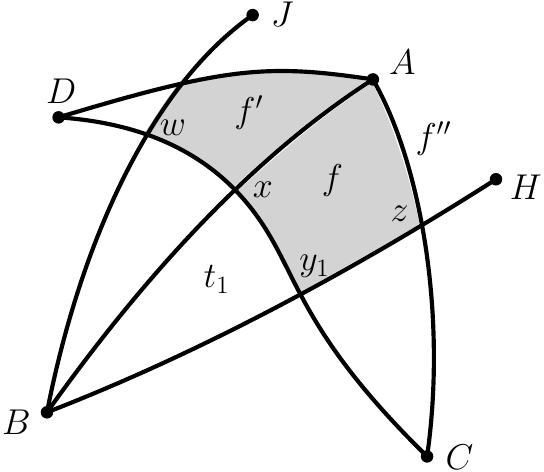}\label{fig:f''-quad}}
	\hspace{5mm}
	\subfloat[If $f'$ is a $1$-triangle and there is an edge that crosses $By_1$ and shares $D$ with $(C,D)$ then its other endpoint must by $A$. This implies that this edges intersects $(A,B)$ twice.]{\includegraphics[width= 7cm]{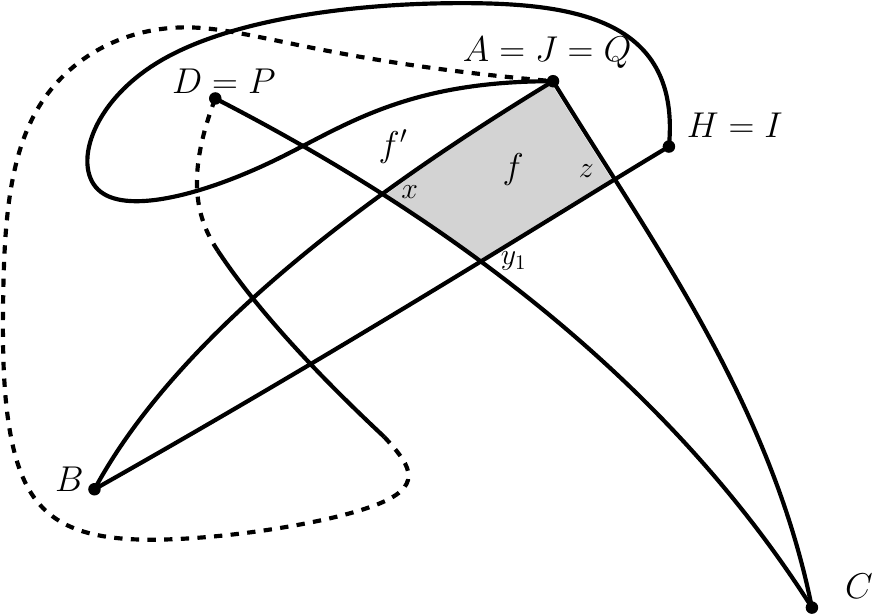}\label{fig:1-quad-1-tri3}}
	\caption{Illustrations for the proof of Propositions~\ref{prop:f'-1quad} and~\ref{prop:f'-1triangle}.}
	\label{fig:after-step5}		
\end{figure}
Since $(B,J)$, $(B,A)$ and $(B,H)$ all cross $(C,D)$, every other edge that crosses $(C,D)$ must also be incident to $B$.
Therefore, it follows from Proposition~\ref{prop:f1-or-distant} that $ch_5(f) \ge 0$.
\end{proof}

\begin{prop}\label{prop:f'-1triangle}
	If $f'$ is a $1$-triangle and $f''$ is neither a $1$-triangle nor a $1$-quadrilateral whose charge after Step~2 is negative, then $ch_5(f) \ge 0$.
\end{prop}

\begin{proof}
Let $(A,I)$ be the edge of $G$ that contains the other edge of $f'$ besides $Ax$ which is incident to $A$.
As in the proof of Proposition~\ref{prop:two-1-triangles}, we conclude that $H=I$ and that no other edges besides $(A,H)$, $(A,B)$ and $(B,H)$ cross $(C,D)$.
Therefore, $Cy_1$ is an edge of $M(G)$ and $t_2$ is an edge-neighbor of $f$ at $wy_1$.
If $By_1$ is also an edge of $M(G)$, then $f_1$, the vertex-neighbor of $f$ at $y_1$ contributes at least $0.2$ units of charge to $f$ in Step~4 by Proposition~\ref{prop:0.2-in-Step-4} and thus $ch_5(f) \ge 0$.

Assume therefore that $By_1$ is crossed by an edge $(P,Q)$ and note that this edge must also cross $Bx$.
Since both $(P,Q)$ and $(C,D)$ cross $(B,H)$ they must share an endpoint.
We claim that this endpoint must be $C$.
Indeed, suppose without loss of generality that $P=D$, then $Q=A$ since $(A,C)$ also crosses $(B,H)$.
But then $(P,Q)$ and $(A,B)$ intersect twice since $(P,Q)$ crosses $Bx$ which is part of $(A,B)$, see Figure~\ref{fig:1-quad-1-tri3} for an illustration.
Therefore, every edge which crosses $(B,H)$ must be incident to $C$.
Thus, it follows from Proposition~\ref{prop:f1-or-distant} that $ch_5(f) \ge 0$.
\end{proof}

Since we have considered all the possible cases with respect to the faces $f'$ and $f''$, we conclude that $ch_5(f) \ge 0$ and the lemma follows.
\end{proof}

It follows from Lemma~\ref{lem:after-step-5} and the discharging steps that the final charge of every face is non-negative and the final charge of every vertex $A \in V(G)$ is $0.4\deg(A)$.
Since the total charge remains $4n-8$ we have
$0.8|E(G)|=\sum_{A \in V(G)} 0.4\deg(v) \leq 4n-8$ and thus $|E(G)| \leq 5n-10$.

This concludes the first part of Theorem~\ref{thm:no-conf-I}.
If $G$ is drawn using straight-line edges, then we may assume without loss of generality that the boundary of the outer face $f_{\rm out}$ is a convex polygon whose vertices are a subset of at least three vertices of $G$.
It follows that $f_{\rm out}$ contributes charge only in Step~1, since in all the other steps the contributing faces have vertices that are crossing points.
Thus $ch_5(f_{\rm out}) = |f_{\rm out}|+|V(f_{\rm out})|-4-0.4|V(f_{\rm out})|=1.6|f_{\rm out}|-4 \ge 0.8$.
Therefore,  we have $0.8+0.8|E(G)|=0.8+\sum_{A \in V(G)} 0.4\deg(v) \leq 4n-8$ and thus $|E(G)| \leq 5n-11$.

\section{Discussion}
We have shown that an $n$-vertex simple adjacency-crossing graph has at most $5n-10$ edges.
The proof of this result could have been simplified if we would have used the results of Brandenburg~\cite{Br20} by which it is enough to consider fan-crossing graphs or even weakly fan-planar graphs. However, we chose to present a somewhat more complicated yet self-contained proof.

For graphs drawn by straight-line edges we have obtained the better bound $5n-11$.
Both of these bounds are tight, as there are matching lower bound constructions in~\cite{KU22}.
Note that it follows from~\cite{KKRS23} that the upper bound on the density of simple (weakly) fan-planar topological graph holds also for non-simple (weakly) fan-planar topological graphs.
It is likely that the same holds also for non-simple fan- or adjacency-crossing graphs.
Furthermore, Kaufmann and Ueckerdt~\cite{KU22} conjectured that the density of simple fan-planar topological graphs with non-homeomorphic parallel edges and non-trivial loops remains $5n-10$.
We believe that our proof methods would imply this result, however, it would require defining additional discharging steps and analyzing several more cases, and therefore for the sake of simplicity we chose not to aim for that more general result in this paper.

\footnotesize
\bibliographystyle{plainurl}
\bibliography{fan-planar}

\end{document}